\documentclass[12pt,a4paper]{amsart}
\usepackage[utf8]{inputenc}
\usepackage{enumerate,bbm,titletoc,a4,mathtools,mathrsfs,amssymb,amsthm}
\usepackage[usenames,dvipsnames]{xcolor}
\usepackage[margin=2.5cm]{geometry}
\usepackage[pagebackref,colorlinks,linkcolor=BrickRed,citecolor=OliveGreen,urlcolor=black,hypertexnames=true]{hyperref}

\DeclareMathOperator{\id}{id}
\DeclareMathOperator{\ev}{ev}

\DeclareMathOperator{\GL}{GL}
\DeclareMathOperator{\Unit}{U}

\DeclareMathOperator{\opm}{M}
\DeclareMathOperator{\oph}{H}

\newcommand{\de}{\delta}

\newcommand{\N}{\mathbb{N}}

\newcommand{\R}{\mathbb{R}}
\newcommand{\C}{\mathbb{C}}

\newcommand{\cB}{\mathcal{B}}
\newcommand{\cC}{\mathcal{C}}
\newcommand{\cD}{\mathcal{D}}

\newcommand{\cK}{\mathcal{K}}

\newcommand{\cZ}{\mathcal{Z}}

\newcommand{\bw}{\mathbf{w}}

\newcommand{\uX}{{\underline X}}
\newcommand{\uY}{{\underline Y}}
\newcommand{\ux}{{\underline x}}

\newcommand*{\mtx}[1]{\opm_{#1}(\C)}
\newcommand*{\her}[1]{\oph_{#1}(\C)}

\newcommand{\Langle}{\mathop{<}\!}
\newcommand{\Rangle}{\!\mathop{>}}

\newcommand{\px}{\C\!\Langle \ux\Rangle}

\DeclareFontFamily{U}{mathx}{}
\DeclareFontShape{U}{mathx}{m}{n}{<-> mathx10}{}
\DeclareSymbolFont{mathx}{U}{mathx}{m}{n}
\DeclareMathAccent{\widehat}{0}{mathx}{"70}
\DeclareMathAccent{\widecheck}{0}{mathx}{"71}

\newtheorem{thm}{Theorem}
\newtheorem{prop}[thm]{Proposition}
\newtheorem{cor}[thm]{Corollary}
\theoremstyle{remark}
\newtheorem{rem}[thm]{Remark}

\title{
Linear matrix pencils and noncommutative convexity
}
\author{Jurij Vol\v{c}i\v{c}}
\address{Department of Mathematics, Drexel University, Pennsylvania}
\email{jurij.volcic@drexel.edu}
\thanks{Supported by the NSF grant DMS-2348720}
\date{\today}



\begin{document}

\begin{abstract}
Hermitian linear matrix pencils are ubiquitous in control theory, operator systems, semidefinite optimization, and real algebraic geometry. This survey reviews the fundamental features of the matricial solution set of a linear matrix inequality, the free spectrahedron, from the perspective of free real algebraic geometry. Namely, among matricial solution sets of noncommutative polynomial inequalities, free spectrahedra are precisely the convex ones. Furthermore, a procedure for detecting free spectrahedra and producing their representing linear matrix pencils is discussed. Finally, free spectrahedra admit a perfect Positivstellensatz, leading to a semidefinite programming formulation of eigenvalue optimization over convex matricial sets constrained by noncommutative polynomial inequalities.
\end{abstract}

\maketitle

\section{Introduction}

A linear matrix pencil is a multivariate affine matrix $L=A_0+A_1x_1+\cdots+A_nx_n$, where $A_j$ are constant square matrices. 
General linear matrix pencils are ubiquitous in matrix theory, numerical analysis, control theory and algebraic geometry. When the coefficients $A_j$ are hermitian matrices, pencils give rise to linear matrix inequalities (LMIs) $L\succeq0$. These are feasible regions of semidefinite programs in mathematical optimization \cite{WSV}, and form a pillar of modern systems engineering \cite{BEFB}.
From the theoretical point of view, hermitian linear matrix pencils are investigated in real algebraic geometry \cite{BCR}, with a special focus on their determinants \cite{HV,Bra,NT}.

More recently, there has been a rising interest in matrix solutions of an LMI.
These are the tuples of hermitian matrices $\uX=(X_1,\dots,X_n)$ such that $L(\uX)=A_0\otimes I+A_1\otimes X_1+\cdots+A_n\otimes X_n$ is positive semidefinite. The collection of all such tuples (of arbitrary sizes) is called the free spectrahedron of $L$. There are at least two ways of looking at free spectrahedra, both responsible for fast advances in their theory and applications. On one hand, free spectrahedra are standard examples of matrix convex sets, and are firmly intertwined with operator systems and completely positive maps in operator algebras and quantum information theory \cite{EW,HKM2,DDSS,Kri,BN}. On the other hand, they are distinguished cases of solution sets of noncommutative polynomial inequalities investigated by free real algebraic geometry and noncommutative optimization
\cite{HMP,HM0,HKM0,BKP,KVV1}.
This survey reviews the interplay between convexity and semialgebraicity in dimension-free matrix variables, and the role of linear matrix pencils in it.

A noncommutative polynomial is a polynomial expression in freely noncommuting variables $x_1,\dots,x_n$. Given a square matrix of noncommutative polynomials $f$, its associated free semialgebraic set $\cD_f$ is the collection of $n$-tuples of hermitian matrices $\uX$ (of same size) such that $f(\uX)$ is positive semidefinite. The adjective ``free'' signifies that the size of matrices in $\uX$ can be arbitrary, and that $x_j$ are free of any inherent relations; on the other hand, ``semialgebraic'' signifies that this set arises from a polynomial inequality (in real algebraic geometry, solution sets of commutative polynomial inequalities are more precisely called basic closed semialgebraic sets).
In particular, free spectrahedra are special cases of free semialgebraic sets. But how special? One can quickly observe that free spectrahedra are convex in the following sense: if $L$ is a linear matrix pencil and $\uX,\uY$ are two matrix tuples of the same size such that $L(\uX),L(\uY)$ are positive semidefinite, then $L(\frac12 \uX+\frac12 \uY)=\frac12 L(\uX)+\frac12 L(\uY)$ is also positive semidefinite. When classical commutative polynomial inequalities are considered, it is not hard to see that there are convex semialgebraic sets that are not solution sets of LMIs. On the contrary, in the dimension-free framework introduced above, Helton and McCullough \cite{HM} showed that free spectrahedra are precisely the convex free semialgebraic sets. This theorem indicates that linear matrix pencils are the sole source of convexity in free real algebraic geometry. On the other hand, it also raises the following question:
given a nonlinear $f$, how does one verify whether $\cD_f$ is convex? After all, $\cD_f$ consists of countably many semialgebraic sets, one for each matrix size, and is therefore a rather large set. A computational criterion for deciding whether $\cD_f$ is a free spectrahedron was derived in \cite{HKMV}.
Among the various benefits of having an LMI representation of a convex free semialgebraic sets, one is the
characterization of noncommutative polynomials that are positive semidefinite on a given free spectrahedron.
This result is called a convex Positivstellensatz \cite{HKM1}, in analogy with sum-of-squares positivity certificates in real algebraic geometry dating back to Hilbert's 17th problem.
However, the conclusions of the convex Positivstellensatz are much stronger than its classical commutative analogs or other noncommutative Positivstellens\"atze.
This survey focuses on the aforementioned three milestones (detection of free spectrahedra, convexity of free semialgebraic sets, convex Positivstellensatz) and related results.

\section{Preliminaries}\label{s:prelim}

Throughout the paper let $\oplus$ and $\otimes$ denote the direct sum and the Kronecker product of matrices, respectively.
Let $\ux=(x_1,\dots,x_n)$ be a tuple of freely noncommuting variables, and let $\px$ be the free $\C$-algebra generated by $x_1,\dots,x_n$. Elements of $\px$, and more generally, matrices over $\px$, are called \emph{noncommutative polynomials}. We also endow $\px$ with the unique skew-linear involution $*$ determined by $x_j^*=x_j$, and extend it to $\opm_d(\px)=\mtx{d}\otimes \px$ as the conjugate transpose on $\mtx{d}$. The \emph{degree} $\deg f$ of $f\in\opm_d(\px)$ is the length of the longest word in $\ux$ that appears with a nonzero coefficient in any entry  of $f$.
The central objects of this survey are affine (degree-one) matrices over $\px$, which are also called \emph{(linear matrix) pencils}, i.e.,
$$L=A_0+A_1x_1+\cdots+A_nx_n \quad\text{where }
A_j\in\mtx{d}.$$
We say that $d$ is the \emph{size} of $L$.
If $A_0=I$ then $L$ is \emph{monic}, and furthermore \emph{irreducible} if $A_1,\dots,A_n$ generate $\mtx{d}$ as an algebra. 

Next, we wish to view noncommutative polynomials as multivariate matrix functions; this perspective is key in free analysis \cite{KVV}.
Given $f\in\opm_d(\px)$ and $\uX=(X_1,\dots,X_n)\in\mtx{k}^n$ let $f(\uX)\in\mtx{dk}$ be the evaluation of $f$ at $\uX$ defined in a natural way. Formally, $f(\uX)$ it is the image of $f$ under the map
$$\opm_d(\px)=\mtx{d}\otimes \px
\xrightarrow{\id\otimes \ev_{\uX}}
\mtx{d}\otimes \mtx{k}=\mtx{dk},$$
where $\ev_{\uX}:\px\to \mtx{k}$ is the unital homomorphism determined by $\ev_{\uX}(x_j)=X_j$.
Let $\her{k}$ denote the real space of hermitian matrices in $\mtx{k}$. If $\uX\in\her{k}^n$, then $\ev_{\uX}$ is a $*$-homomorphism, and $f(\uX)\in\her{dk}$ for all $f=f^*\in\opm_d(\px)$.

We can now consider noncommutative polynomial inequalities.
A \emph{free semialgebraic set} \cite{HM0,HKM0} is the \emph{positivity domain} of a hermitian noncommutative polynomial $f=f^*\in\opm_d(\px)$,
$$\cD_f=\bigcup_{k\in\N} \cD_f(k),\quad 
\cD_f(k)=\{\uX\in\her{k}^n\colon f(\uX)\succeq0\}. 
$$
Note that $\cD_{f_1\oplus\cdots\oplus f_\ell}=\cD_{f_1}\cap\cdots\cap \cD_{f_\ell}$.
Free semialgebraic sets are therefore dimension-independent noncommutative analogs of basic closed semialgebraic sets in real algebraic geometry \cite{BCR}. 
If $L$ is a hermitian pencil (i.e., all its coefficients are hermitian matrices), then $\cD_L$ is called a \emph{free spectrahedron} or an \emph{LMI domain}, while $L$ is also called an \emph{LMI representation} of the set $\cD_L$. This terminology stems from the fact that $\cD_L(1)\subseteq\R^n$ is a spectrahedron in the classical sense \cite{Vin}.
Two standard examples of free spectrahedra (for $r\in\R_{>0}$) are the \emph{free cube}
$$\cC_r =\cD_C,\qquad C=\begin{pmatrix}
r& x_1\\ x_1 & r
\end{pmatrix}\oplus\cdots\oplus
\begin{pmatrix}
r& x_n\\ x_n & r
\end{pmatrix},
$$
consisting of hermitian tuples with spectra in $[-r,r]$, and the \emph{free ball}
$$\cB_r =\cD_B,\qquad B=\begin{pmatrix}
1&&&x_1\\
&\ddots&&\vdots\\
&&1&x_n\\
x_1&\cdots&x_n&r^2
\end{pmatrix},
$$
consisting of hermitian tuples $\uX$ such that the eigenvalues of $X_1^2+\cdots+X_n^2$ are at most $r^2$ (the unspecified entries in the pencil are zero).

Certain features of a free spectrahedron $\cD_L$ are determined by the (classical) spectrahedron $\cD_L(1)$.
\begin{prop}\label{p:level1}
Let $L$ be a hermitian pencil, and $\ell=
\alpha_0+\alpha_1x_1+\cdots+\alpha_nx_n 
\in\R\!\Langle \ux\Rangle$.
\begin{enumerate}[(a)]
\item $\cD_L\neq\emptyset$ if and only if $\cD_L(1)\neq\emptyset$.
\item $\ell|_{\cD_L}=0$ if and only if $\ell|_{\cD_L(1)}=0$.
\item $\cD_L\subseteq \cC_r$ if and only if $\cD_L(1)\subseteq \cC_r(1)$.
\end{enumerate}
\end{prop}

\begin{proof}
In each of the asserted statements, only one implication requires a justification. For $\uX\in\her{k}^n$ and $v\in\C^k$ let $v^*\uX v=(v^*X_1v,\dots,v^*X_nv)\in\R^n$.
\\
($\Rightarrow$) in (a): let $\uX\in \cD_L(k)$. Then for every unit vector $v\in \C^k$ we have $L(v^*\uX v)=(I\otimes v)^*L(\uX)(I\otimes v)\ge0$, and thus $\cD_L(1)\neq0$.
\\
($\Leftarrow$) in (b): suppose $\ell(\uX)\neq0$ for $\uX\in \cD_L(k)$. Then there exists a unit vector $v\in \C^k$ such that $\ell(v^*\uX v)=v^*\ell(\uX)v\neq0$. Since $v^*\uX v\in\cD_L(1)$ as seen in the proof of (a), it follows that $\ell$ does not vanish on $\cD_L(1)$.
\\
($\Leftarrow$) in (c): suppose $\cD_L(k)\not\subseteq \cC_r(k)$ for some $k\in\N$. Then there are $\uX\in\cD_L(k)$ and a unit vector $v\in\C^k$ such that $v^*X_jv \notin[-r,r]$ for some $j\in\{1,\dots,n\}$. Therefore $v^*\uX v\in \cD_L(1)\setminus \cC_r(1)$.
\end{proof}

Let $L$ be a hermitian pencil.
The properties of $\cD_L(1)$ in Proposition \ref{p:level1} can be handled with \emph{commutative} real algebra. In particular, \cite[Theorem 4.3.3]{KS} gives a computational certificate for $\cD_L(1)$ being non-empty, while \cite[Theorem 3.4.1]{KS} allows one to either assert that $\cD_L(1)$ has nonempty interior, or determine the affine hull of $\cD_L(1)$ in $\R^n$.
Furthermore, \cite[Corollary 4.4.2]{KS} gives an algebraic certificate for $\cD_L(1)$ being bounded.
In view of Proposition \ref{p:level1} and the aforementioned results on classical spectrahedra from \cite{KS}, 
it is for various purposes sufficient to assume that $L$ is \emph{monic}, i.e., $L=I+A_1x_1+\cdots+A_nx_n$. Indeed: by looking at $\cD_L(1)$, one can restrict to its affine hull and apply a suitable change of coordinates to ensure that $0$ lies in the interior of $\cD_L$. By \cite[Proposition 2.1]{HKM2}, one can then effectively replace $L$ with a monic hermitian pencil.
For this reason, many statements on free spectrahedra in the literature adopt the monic property as a standard assumption.

\section{Detection of free spectrahedra}\label{s:det}

Let $f$ be a hermitian noncommutative polynomial with $f(0)\succ0$. This section presents a procedure for detecting whether $\cD_f$ is a free spectrahedron, and for producing its LMI representation.
This procedure is based on a block-triangular form of pencils and their determinantal varieties, to which we turn next.

Two monic pencils $L,M \in\opm_d(\px)$ are \emph{similar} or \emph{unitarily similar} if $M=ULU^{-1}$ for $U\in\GL_n(\C)$ or $U\in\Unit_n(\C)$, respectively.
In general, a monic pencil $L$ is similar to
\begin{equation}\label{e:triang}
\begin{pmatrix}
L_1 & \star &\cdots &\star \\
0 & \ddots & \ddots & \vdots \\
\vdots & \ddots & \ddots & \star \\
0 & \cdots & 0 & L_\ell
\end{pmatrix}
\end{equation}
where each of $L_1,\dots,L_\ell$ is either 1 or an irreducible pencil,
by Burnside's theorem \cite[Corollary 5.23]{Bre} on the existence of common invariant subspaces of the coefficients of $L$.

Given $f\in\opm_d(\px)$ let
$$\cZ_f=\bigcup_{k\in\N} \cZ_f(k),\quad 
\cZ_f(k)=\{\uX\in\mtx{k}^n\colon \det f(\uX)=0\}$$
be its \emph{free locus} \cite{KV,HKV}. Free loci represent the geometric counterpart of factorization in the free algebra \cite{HKV1}. In particular, they carry the essential information about irreducible monic pencils, as follows.

\begin{thm}[{\cite[Theorem 3.11, Corollary 5.5]{KV} and \cite[Theorems C and E]{HKV}}]\label{t:irr}
Let $L$ be an irreducible monic pencil.
\begin{enumerate}[(a)]
    \item Any (hermitian) irreducible monic pencil whose free locus equals $\cZ_L$ is (unitarily) similar to $L$.
    \item $\cZ_L(k)$ is an irreducible hypersurface in $\mtx{k}^n$ for all but finitely many $k\in\N$.
    \item If $L$ is hermitian, then the boundary of $\cD_L(k)\subseteq\her{k}^n$ is Zariski dense in $\cZ_L(k)\subseteq\mtx{k}^n$ for all but finitely many $k\in\N$.
\end{enumerate}
\end{thm}

For the purpose of this survey it suffices to consider free loci of monic pencils, because for every noncommutative polynomial $f$ with $\det f(0)\neq0$ there exists a monic pencil $L\in \opm_d(\px)$ such that $\cZ_L=\cZ_f$. Such an $L$ is called a \emph{linearization} of $f$. A very direct way of finding a linearization of $f$ is to apply Higman's trick \cite[Section 5.8]{Coh}
$$\cZ_{a+bc}=\cZ_{\left(\begin{smallmatrix}
a&-b\\
c&I
\end{smallmatrix}\right)}, \qquad
a\in \px\!{}^{\alpha\times\alpha},\ 
b\in \px\!{}^{\alpha\times\beta},\ 
c\in \px\!{}^{\beta\times\alpha},
$$
repeatedly to monomials in $f$, until only terms of degree at most 1 are left in every entry.
Alternatively, as in \cite{HKMV}, one can apply the realization theory for noncommutative rational functions from control theory \cite{BGM} to $f^{-1}$ in order to obtain a linearization $L$ of $f$. This approach in particular allows for removing the redundancies in the size of $L$ using linear-algebraic tools. In any case, a linearization $L$ of $f$ as defined above is not unique; one can always choose a basis such that $L$ is block-triangular as in \eqref{e:triang}, with irreducible blocks on the diagonal.

By Theorem \ref{t:irr}(c), the positivity domain of an irreducible hermitian pencil $L$ determines its free locus, while the latter determines $L$ by Theorem \ref{t:irr}(a). Also, the free locus of \eqref{e:triang} detects the free loci of its irreducible diagonal blocks through its hypersurface decomposition by Theorem \ref{t:irr}(b). These observations are key to the following criterion for recognizing free spectrahedra.

\begin{thm}[{\cite[Theorem 1.1]{HKMV}}]\label{t:comp}
Let $f$ be a hermitian noncommutative polynomial with $f(0)\succ0$, and $L$ its linearization. After a basis change, one can assume that $L$ is in block-triangular form \eqref{e:triang}, with irreducible blocks on the diagonal that are either hermitian or not similar to hermitian pencils. Let $\widehat{L}$ be the direct sum of hermitian irreducible blocks, and let $\widecheck{L}$ be the direct sum of the remaining irreducible blocks.

Then $\cD_f$ is a free spectrahedron if and only if $\cZ_{\widecheck{L}}\cap \cD_{\widehat{L}}\subseteq \cZ_{\widehat{L}}$, in which case $\cD_f=\cD_{\widehat{L}}$.
\end{thm}

Theorem \ref{t:comp} gives rise to the following
computationally feasible procedure for determining whether a free semialgebraic set is a free spectrahedron. Let $f$ be a hermitian noncommutative polynomial with $f(0)\succ0$.

\begin{enumerate}
    \item Find a linearization $L$ of $f$ using Higman's trick or realization theory for $f^{-1}$; then determine the irreducible blocks $L_j$ appearing on the diagonal of the block-triangular form of $L$ (see \cite{CIK} for a polynomial-time algorithm to extract the $L_j$).
    \item Decide which of the $L_j$ are similar to hermitian pencils. This can be done by solving the semidefinite program
    $$Q\succeq I,\quad QL_j^*=L_jQ.$$
    If no such $Q$ exist, then $L_j$ is not similar to a hermitian pencil; if such a $Q$ exists, then $\sqrt{Q}^{-1}L_j\sqrt{Q}$ is hermitian.
    The irreducible blocks that are similar to hermitian pencils can then be replaced by hermitian blocks.
    \item Set $\widehat{L}$ to be the direct sum of hermitian $L_j$, and let $\widecheck{L}$ be the direct sum of the remaining $L_j$. One can further reduce the size of $\widehat{L}$ using the algorithm in \cite[Subsection 4.6.1]{HKM2}, again in the form of a semidefinite program, to remove redundant blocks (that do not affect $\cD_{\widehat{L}}$).
    \item Verify whether the condition $\cZ_{\widecheck{L}}\cap \cD_{\widehat{L}}\subseteq \cZ_{\widehat{L}}$ is satisfied.
    While not quite straightforward, this condition can also be verified through a finite sequence of feasibility semidefinite programs; see \cite[Algorithm 4.3.1]{HKMV}.
\end{enumerate}

Step (3) in the above procedure refers to possible redundancy in the LMI representation describing a free spectrahedron. This is not unexpected; for example, if $L_1,L_2$ are such that $\cD_{L_1}\subseteq\cD_{L_2}$, then the free spectrahedron $\cD_{L_1}=\cD_{L_1\oplus L_2}$ admits LMI representations $L_1$ and $L_1\oplus L_2$.
We say that a hermitian pencil $L$ is \emph{minimal} if it has the minimal size among all hermitian pencils representing $\cD_L$. The following corollary shows that minimal LMI representations of free spectrahedra are essentially unique.

\begin{cor}\label{c:unique}
Let $L$ and $M$ be minimal monic hermitian pencils. Then $\cD_L=\cD_M$ if and only if $L$ and $M$ are unitarily similar.
\end{cor}

Corollary \ref{c:unique} was first established in \cite[Theorem 1.2]{HKM2} under the assumption that $\cD_L$ is bounded. This assumption was relaxed in \cite[Section 6.2]{DDSS}, and completely removed in \cite[Theorem 3.1]{Zal}. Alternatively, Corollary \ref{c:unique} also follows from Theorem \ref{t:irr}; this approach was taken in \cite[Corollary 5.17]{Kri}. Indeed, every minimal monic hermitian pencil is unitarily similar to a direct sum of pairwise non-similar irreducible hermitian pencils $L=L_1\oplus\cdots\oplus L_\ell$, such that $\bigcap_{k\neq j}\cD_{L_k}\not\subseteq \cD_{L_j}$ for all $j$. From here it is not hard to deduced that for all large enough $k$, the intersection of the boundary of $\cD_{L}(k)$ and the boundary of $\cD_{L_j}(k)$ has real codimension 1, for every $j$. Thus $\cZ_L=\bigcup_j\cZ_{L_j}$ by Theorem \ref{t:irr}(c), so $L$ is unique up to unitary similarity by Theorem \ref{t:irr}(a).

\section{Convex free semialgebraic sets}\label{s:cvx}

This section addresses convexity of free semialgebraic sets, which is the distinguishing geometric property of LMI domains.

A family of sets $\cK=(\cK(k))_{k\in\N}$ with $\cK(k)\subseteq \her{k}^n$ is \emph{convex} if $\cK(k)$ is a convex subset of the real affine real space $\her{k}^n$, for every $k\in\N$. 
There is also a stronger notion of convexity for such families of sets, which considers not only convexity of each set by itself, but also the relationship between the sets.
We say that $\cK$ is \emph{matrix convex} if 
for all $k_1,\dots,k_\ell,k\in\N$
and $V_j\in\C^{k_j\times k}$ such that $V_1^*V_1+\cdots+V_\ell^*V_\ell=I$,
$$\uX^{(1)}\in \cK(k_1),\dots, \uX^{(\ell)}\in\cK(k_\ell) \quad\implies\quad
V_1^*\uX^{(1)}V_1+\cdots+V_\ell^*\uX^{(\ell)}V_\ell \in\cK(k).
$$
Here, $V^*\uX V=(V^*X_1V,\dots,V^*X_nV)$. In particular, if $\cK$ is matrix convex then it is convex.
Observe that free spectrahedra are examples of matrix convex sets. Indeed, if $L$ is a hermitian pencil and $\uX^{(j)}\in\cD_L(k_j)$ and $V\in\C^{k_j\times k}$ are as above, then
\begin{align*}
&L\Big(V_1^*\uX^{(1)}V_1+\cdots+V_\ell^*\uX^{(\ell)}V_\ell\big) \\
=\,& (I\otimes V_1)^*L(\uX^{(1)})(I\otimes V_1)
+\cdots+
(I\otimes V_\ell)^*L(\uX^{(\ell)})(I\otimes V_\ell)
\succeq\, 0,
\end{align*}
so $\cD_L$ is matrix convex, and thus convex.
The following theorem by Helton and McCullough is a remarkable converse of this observation.

\begin{thm}[{\cite[Theorem 1.4]{HM}}]\label{t:cvx}
Let $f=f^*\in\opm_d(\px)$ and $f(0)\succ0$.
Then $\cD_f$ is convex if and only if it is a free spectrahedron.
\end{thm}

In particular, convexity and matrix convexity coincide for free semialgebraic sets defined by noncommutative polynomials that are positive definite at some scalar point. As opposed to Section \ref{s:det}, which presented a (rather intricate) algebraic criterion for recognizing free spectrahedra, Theorem \ref{t:cvx} shows that convexity, a purely geometric property, is the distinguishing feature of free spectrahedra among free semialgebraic sets. Combining Theorems \ref{t:comp} and \ref{t:cvx} thus gives us a computationally feasible scheme for verifying convexity of free semialgebraic sets.
Furthermore, note that the classical (commutative) analog of Theorem \ref{t:comp} fails: for example, $\{(x_1,x_2)\colon x_1^4+x_2^4\le 1\}$ is convex in $\R^2$ but not a spectrahedron (cf. \cite{HV}).

To be precise, the proof in \cite[Theorem 1.4]{HM} of the forward implication of Theorem \ref{t:cvx} rests on a further assumption that $\cD_f$ is bounded (i.e., $\cD_f(1)$ is bounded, cf. Proposition \ref{p:level1}).
On the other hand, \cite[Theorem 1.4]{HM} is also technically stronger than Theorem \ref{t:irr}, due to a slightly different definition of $\cD_f$ in \cite{HM} (namely, as a connected component of the invertibility set of $f$; the definition of $\cD_f$ in Section \ref{s:prelim} is used in this survey for the sake of simplicity).
At the end of the section we give an argument why the boundedness assumption is not required in Theorem \ref{t:irr}.

The proof of Theorem \cite[Theorem 1.4]{HM} consists of three parts. Firstly, it is established that a convex free semialgebraic set $\cD_f$ with $f(0)\succ0$ is a matrix convex set. Secondly, the Effros-Winkler theorem \cite{EW}, the analog of the Hahn-Banach separation theorem for matrix convex sets, then implies that $\cD_f$ is the intersection of infinitely many free spectrahedra. Thirdly, an intricate analysis of the boundary of $\cD_f$ (more precisely, the behavior of kernels of $f(\uX)$ for $\uX$ in the boundary of $\cD_f$) and a modification of the Effros-Winkler theorem (leveraging that $\cD_f$ is not a general matrix convex set, but is bounded and has a rather Noetherian nature since it is given by a noncommutative polynomial) show that $\cD_f$ is the intersection of \emph{finitely} many free spectrahedra, and thus a free spectrahedron. An alternative to this last part is given in \cite[Theorem 4.2]{Kri} using Nash manifolds from real algebraic geometry.

Let us now show that \cite[Theorem 1.4]{HM} (i.e., Theorem \ref{t:cvx} assuming boundedness of the free semialgebraic set) and Theorem \ref{t:comp} imply that Theorem \ref{t:cvx} holds without the boundedness assumption.

\begin{proof}[From {\cite[Theorem 1.4]{HM}} to Theorem \ref{t:cvx}]
Let $f=f^*\in\opm_d(\px)$ be such that $f(0)\succ0$ and $\cD_f$ is convex, but not necessarily bounded. Let $L,\widehat{L},\widecheck{L}$ be as in Theorem \ref{t:comp}. Suppose $\cD_f$ is not a free spectrahedron; then by Theorem \ref{t:comp} there exists
$$\uX\in\cZ_{\widecheck{L}}(k)\cap \cD_{\widehat{L}}(k),\qquad \uX\notin \cZ_{\widehat{L}}(k).$$
Let $r>0$ be such that eigenvalues of $X_1^2+\cdots +X_n^2$ are not larger than $r^2$, and let
$$B=\begin{pmatrix}
1&&&x_1\\
&\ddots&&\vdots\\
&&1&x_n\\
x_1&\cdots&x_n&r^2
\end{pmatrix}$$
be the pencil that determines the free ball $\cB_r$ from Section \ref{s:prelim}. Then $(f\oplus B)(0)\succ0$, and
$\cD_{f\oplus B}=\cD_f\cap \cB_r$ is convex and bounded. By Theorem \cite[Theorem 1.4]{HM} it follows that $\cD_{f\oplus B}$ is a free spectrahedron.
On the other hand, $B$ is an irreducible hermitian pencil, $B(\uX)\succeq0$, and $L\oplus B$ is a linearization of $f\oplus B$ in a block-triangular form as in \eqref{e:triang}. Since
$$\uX\in\cZ_{\widecheck{L}}(k)\cap \cD_{\widehat{L}\oplus B}(k),\qquad \uX\notin \cZ_{\widehat{L}}(k),$$
if follows by Theorem \ref{t:comp} that $\cD_{f\oplus B}$ is not a free spectrahedron, a contradiction. Therefore $\cD_f$ is a free spectrahedron.
\end{proof}

For a further analysis of free spectrahedra from the matrix convexity perspective, and in particular their extreme and exposed points and faces, see \cite{EHKM,Kri,EH,KSt}.

\section{Positivstellensatz for free spectrahedra}\label{s:pos}

In this section, we consider yet another crucial property of free spectrahedra that is rather exceptional among general free semialgebraic sets. Namely, positive semidefiniteness of noncommutative polynomials on a free spectrahedron admits an effective sums-of-squares certificate. Since all convex free semialgebraic sets (with nonempty interior) are free spectrahedra by Theorem \ref{t:cvx}, this algebraic certificate of positivity is called a \emph{convex Positivstellensatz}.

\begin{thm}[{\cite[Theorem 1.1]{HKM1}}]\label{t:pos}
Let $f=f^*\in\opm_d(\px)$, and let $L$ be a monic hermitian pencil of size $e$. Then $f\succeq0$ on $\cD_L$ if and only if
\begin{equation}\label{e:pos}
f=\sum_{i=1}^M s_i^*s_i+\sum_{j=1}^N v_j^*L v_j
\end{equation}
for some $s_i\in\opm_d(\px),\ v_j\in\px^{e\times d}$ with 
$\deg s_i,\deg v_j \le \lfloor \frac{\deg f}{2}\rfloor$.
\end{thm}

It is hard to overstate the significance of Theorem \ref{t:pos}. 
Firstly, its special case for $L=1$ (proven independently by Helton \cite{Hel} and Mccullough \cite{McC}) describes globally positive semidefinite noncommutative polynomials, and can thus be viewed as the resolution of a free analog of Hilbert's 17th problem. Secondly, it is well-known that there is no comparable analog of Theorem \ref{t:pos} for (classical) spectrahedra in real algebraic geometry.
Thirdly, other Positivstellens\"atze for noncommutative polynomials either offer weaker certificates (e.g., the Positivstellensatz for bounded free semialgebraic sets \cite{HM0}, which involves constant term perturbations) or have much more modest reach (e.g., Positivstellens\"atze for spherical isometries and noncommutative unitary groups \cite{HMP,KVV1}).
Finally, the existence of degree bounds in the sums-of-square representation \eqref{e:pos} renders this certificate computationally feasible via semidefinite programming (see Subsection \ref{s:opt} below for details). The number of terms in \eqref{e:pos} can also be bounded by $M+N\le 1+d^2\sum_{j=0}^{\deg f+1}n^j$ using Carath\'eodory's theorem on convex hulls (see Remark \cite[Remark 1.2]{HKM1}).

\begin{rem}
The name ``convex Positivstellensatz'' for Theorem \ref{t:pos} is justified by Theorem \ref{t:cvx}, which guarantees existence of LMI representations for convex free semialgebraic sets. However, the use of an LMI representation in Theorem \ref{t:pos} is essential. For example, if $f=1-x_1$, then $\cD_{f^3}=\cD_f$ is a free spectrahedron, but $f$ cannot be written as $\sum_i s_i^*s_i+\sum_j v_j^*f^3 v_j$ for $s_i,v_j\in\px$. Note that this assertion reduces to a calculation in the polynomial ring $\R[x_1]$; suppose $f=\sigma_0+f^3 \sigma_1$ where $\sigma_0,\sigma_1$ are sums of squares of polynomials in $\R[x_1]$. Then one concludes that $\sigma_0$ is divisible by $f$, and therefore by $f^2$ since $\sigma_0$ is a sum of squares. But then $1=f(\frac{\sigma_0}{f^2}+f^2 \sigma_1)$, a contradiction.
\end{rem}

\begin{rem}
It is also important that $L$ in Theorem \ref{t:pos} is monic (or at least has a positive definite constant term). For example, consider the hermitian pencil
$$L=\begin{pmatrix}
x_2&x_1\\x_1&0
\end{pmatrix}.$$
Then $\cD_L=\{(0,X_2)\colon X_2\succeq0\}$, so $x_1\succeq0$ on $\cD_L$. Suppose $x_1=\sum_i s_i^*s_i+\sum_j v_j^*L v_j$ for noncommutative polynomials $s_i,v_j$. Since $x_1$ and $L$ are both linear, it follows that $s_i(0)=0$ for all $i$, and then $x_1=\sum_j v_j(0)^*L v_j(0)$, which is clearly impossible.
\end{rem}

The proof of the forward implication in Theorem \ref{t:pos} (the converse is straightforward) is based on a truncated version of the Gelfand-Naimark-Segal construction. For $\de\in\N$ let $\opm_d(\px)_\de$ denote the subspace of noncommutative polynomials in $\opm_d(\px)$ of degree at most $\de$. If $f$ is not of the form \eqref{e:pos}, one can find a functional $\lambda$ on $\opm_d(\px)_{2\de+2}$ for $\de=\lfloor \frac{\deg f}{2}\rfloor$ such that $\lambda(f)<0$ and $\lambda$ is positive at all nonzero elements of the form as on the right-hand side in \eqref{e:pos}. Such a functional gives rise to an inner product $\langle p,q\rangle=\lambda(p^*q)$ on $\opm_d(\px)_{\de+1}$. Left multiplication by $x_1,\dots,x_n$ on $\opm_d(\px)$ can then be truncated (or compressed) to self-adjoint operators $X_1,\dots,X_n$ on the finite-dimensional inner product space $\opm_d(\px)_\de$. Using the relationship between the inner product and the functional $\lambda$, one can then show that $L(\uX)\succeq0$, while $f(\uX)$ is not positive semidefinite due to $\lambda(f)<0$.

Variants of the Gelfand-Naimark-Segal construction like the one described above are ubiquitous in free real algebraic geometry. In particular, they extend Theorem \ref{t:pos} to a Positivstellensatz for \emph{noncommutative rational functions} that are positive semidefinite on free spectrahedra \cite{Pas,Vol}. In a different direction, techniques behind Theorems \ref{t:pos} and \ref{t:irr} also lead to a characterization of noncommutative polynomials that define a given free spectrahedron and its boundary. This conceptual reverse of Theorem \ref{t:comp} is called a \emph{Randstellensatz}.

\begin{thm}[{\cite[Theorem 1.1]{HKN} and \cite[Corollary 8.7]{HKV}}]\label{t:rand}
Let $f=f^*\in\opm_d(\px)$, and let $L$ be a monic hermitian pencil of size $d$, minimal for $\cD_L$. Then $f\succeq0$ on $\cD_L$ and $\ker L(\uX)\subseteq \ker f(\uX)$ for all $\uX\in\cD_L$ if and only if
$$
f=L\left(\sum_i p_ip_i^*\right)L+\sum_j (q_jL+C_j)^*L(q_jL+C_j)
$$
for some $p_i\in\px^{d\times 1},\ q_j\in\opm_d(\px)$ and $C_j\in\mtx{d}$ satisfying $C_jL=LC_j$.
\end{thm}

\subsection{Noncommutative polynomial optimization on free spectrahedra}\label{s:opt}

An important aspect of Theorem \ref{t:pos} is its application to noncommutative optimization. 
Consider the optimization problem
\begin{equation}\label{e:opt}
\mu_\star=\sup_{\uX\in \cD_L} \, \big[\text{largest eigenvalue of }f(\uX)\big]
\end{equation}
for a given noncommutative polynomial $f$ and a monic hermitian pencil $L$. At first glance, the problem \eqref{e:opt} might seem unwieldy, as the domain of optimization consists of countably many (possibly unbounded) convex sets of increasing dimensions.
However, Theorem \ref{t:pos} implies that the solution to \eqref{e:opt} equals
\begin{equation}\label{e:opt1}
\begin{split}
\mu_\star=\inf_{\mu\in\R}\,\mu\quad \text{such that}\quad
&\mu-f=\sum_i s_i^*s_i+\sum_j v_j^*L v_j \\
&\text{for some }s_i\in\opm_d(\px),\, v_j\in\px{}^{e\times d}\\ 
&\text{with }\deg s_i,\,\deg v_j\le \lfloor\tfrac{f}{2}\rfloor.
\end{split}
\end{equation}
Let us describe \eqref{e:opt1} more concretely.
Let $L$ be a monic hermitian pencil of size $e$, and let $f\in\opm_d(\px)$. Let $\bw$ be a column vector listing all words made from $\ux$ of length at most $\lfloor\frac{\deg f}{2}\rfloor$; then $\bw$ has $\Delta=\sum_{j=0}^{\lfloor\frac{\deg f}{2}\rfloor} n^\delta$ entries. Then \eqref{e:opt1} can be rewritten as
\begin{equation}\label{e:opt2}
\begin{split}
\mu_\star=\inf_{\substack{\mu\in\R,\\
S_0,S_{ij}\in\mtx{d\Delta}\\
\text{for }1\le i,j\le e
}}\,\mu\quad \text{such that}\quad
&\mu-f=(I_d\otimes \bw)^*\left(S_0+\sum_{i,j=1}^e S_{ij}L_{ij}\right)(I_d\otimes \bw), \\
&S_0^*=S_0,\ S_{ij}^*=S_{ji},\\
&S_0\succeq0,\, (S_{ij})_{i,j=1}^e\succeq0.
\end{split}
\end{equation}
Here, $L_{ij}\in \px$ denotes the $(i,j)$ entry of $L$.
It is evident that \eqref{e:opt2} is a semidefinite program. Therefore one can optimize eigenvalues of noncommutative polynomials on convex free semialgebraic sets using the powerful numerical methods from semidefinite programming.
This paradigm is far-reaching; for example, \cite{BKP} presents a general framework for eigenvalue and trace optimization of noncommutative polynomials.


\end{document}